\documentclass[a4paper]{article}

\usepackage{amsfonts}
\usepackage{amsmath}
\usepackage{amssymb}
\usepackage{amsthm}
\usepackage[mathcal]{euscript}
\usepackage{ifthen}
\usepackage{tikz}
\usepackage{enumerate}
\usepackage{mathtools}

\newtheorem{theorem}{Theorem}[section]
\newtheorem{proposition}[theorem]{Proposition}
\newtheorem{corollary}[theorem]{Corollary}
\newtheorem{lemma}[theorem]{Lemma}
\theoremstyle{remark}

\theoremstyle{definition}
\newtheorem{definition}[theorem]{Definition}

\newcommand{\parder}[3][Default]{
	\frac{\partial \ifthenelse{\equal{#1}{Default}}{}{^{#1}}#2}{
              \partial #3 \ifthenelse{\equal{#1}{Default}}{}{^{#1}}}}

\newcommand{\F}{{\mathbb F}}
\newcommand{\tp}{^{\rm t}}
\newcommand{\Mat}{\operatorname{Mat}}
\newcommand{\GL}{\operatorname{GL}}
\newcommand{\tr}{\operatorname{tr}}
\newcommand{\rk}{\operatorname{rk}}
\newcommand{\chr}{\operatorname{char}}

\newcommand{\Co}{{\mathcal C}}
\newcommand{\Ea}{{\mathcal E}}
\newcommand{\Ma}{{\mathcal M}}
\newcommand{\Na}{{\mathcal N}}
\newcommand{\Ve}{{\mathcal V}}
\newcommand{\rd}{{\mathfrak r}}
\newcommand{\A}{{\mathcal A}}
\newcommand{\rct}{\mbox{\begin{tikzpicture}[x=1pt,y=1pt]
  \useasboundingbox (0,0) -- (10,9);
  \draw[color=lightgray] (9,0) -- (0,9);
  \draw (0,0) rectangle (9,9);
  \draw[fill=lightgray] (4,5) rectangle (9,9);
\end{tikzpicture}}_r}
\newcommand{\rctt}{\mbox{\begin{tikzpicture}[x=1pt,y=1pt]
  \useasboundingbox (0,0) -- (10,9);
  \draw[color=lightgray] (9,0) -- (0,9);
  \draw (0,0) rectangle (9,9);
  \draw[fill=lightgray] (0,0) rectangle (5,4);
\end{tikzpicture}}_{n-r}}

\makeatletter
\newcommand{\nolisttopbreak}{\vspace{\topsep}\nobreak\@afterheading}
\makeatother

\newenvironment{listproof}[1][\proofname]{\begin{proof}[#1]\mbox{}\nolisttopbreak}{\end{proof}}

\title{Mathieu subspaces of codimension less than $n$ of $\Mat_n(K)$}
\author{Michiel de Bondt}

\begin{document}

\maketitle

\begin{abstract}
\noindent
We classify all Mathieu subspaces of $\Mat_n(K)$ of codimension less than $n$, 
under the assumption that $\chr K = 0$ or $\chr K \ge n$. 

More precisely, we show that any proper Mathieu subspace of $\Mat_n(K)$ of codimension less 
than $n$ is a subspace of $\{M \in \Mat_n(K) \mid \tr M = 0\}$ if 
$\chr K = 0$ or $\chr K \ge n$. On the other hand, we show that every subspace of 
$\{M \in \Mat_n(K) \mid \tr M = 0\}$ of codimension less than $n$ in $\Mat_n(K)$
is a Mathieu subspace of $\Mat_n(K)$ if $\chr K = 0$ or $\chr K \ge n+1$. 
\end{abstract}



\section{Introduction}

The notion of Mathieu subspaces has been introduced by W. Zhao in \cite{MR2586998}.
The usefulness of this notion has been proven by the many notorious open problems
that has been formulated in terms of it. For more information about Mathieu subspaces 
in general, see \cite{MR2586998}, \cite{MR2859886}, \cite{MR2736912} and the references 
therein. See also \cite{1006.5801} for the connection between Mathieu subspaces and the 
Image Conjecture.

In this paper, we study Mathieu subspaces over a field $K$ of $\Mat_n(K)$: the $n$-dimensional 
matrix ring over $K$. But let us first give the general definition of Mathieu subspaces.

\begin{definition}[following {\cite[Def.\@ 1.1]{MR2859886}} and {\cite[Def.\@ 1.2]{MR2859886}}] 
Let $M$ be an $R$-subspace ($R$-sub\-mod\-ule) of an associative $R$-algebra $\A$.
Then we call $M$ a {\em $\vartheta$-Mathieu subspace} of $\A$ if the following 
property holds for all $a, b, c \in \A$ such that $a^m \in M$ for all $m \ge 1$:
\begin{enumerate}[{\upshape (i)}]
\item $ba^m \in M$ when $m \gg 0$, if $\vartheta =$ ``{\it left}'';
\item $a^mc \in M$ when $m \gg 0$, if $\vartheta =$ ``{\it right}'';
\item $ba^m, a^mc \in M$ when $m \gg 0$, if $\vartheta =$ ``{\it pre-two-sided}'';
\item $ba^mc \in M$ when $m \gg 0$, if $\vartheta =$ ``{\it two-sided}''.
\end{enumerate}
\end{definition}

\noindent
As you can see, there are four different types of Mathieu subspaces. However, for 
Mathieu subspaces over a field $K$ of $\Mat_n(K)$, it is a nice exercise to show that 
the last two types coincide, e.g.\@ by using theorem \ref{rad} in the last section.
In this last section, we look at radicals of Mathieu subspaces. Those radicals play an important 
role in the study of Mathieu subspaces, see the references in the first paragraph of the current 
section.

But the radical can be taken for any subset $S$ of the whole space $\A$, 
not only for Mathieu subspaces. The {\em radical} $\rd(S)$ of $S$ is just the set 
$\{a \in \A \mid a^m \in S$ for all $m \gg 0\}$.

In \cite[Th.\@ 5.1]{MR2859886}, Zhao classifies all Mathieu subspaces of codimension one
of $\Mat_n(K)$. Zhao proved that the subspace $H$ of $\Mat_n(K)$ consisting of all matrices 
with trace zero is the only candidate, and that $H$ is indeed a $\vartheta$-Mathieu subspace
of $\Mat_n(K)$, if and only if $\chr K = 0$ or $\chr K \ge n+1$.

A. Konijnenberg proved in his Master's thesis \cite{konijnenbergmt} that for Mathieu subspaces of 
codimension two of $\Mat_n(K)$, $K$-subspaces of $H$ are the only possible candidates, under the 
assumption that $n \ge 3$, see \cite[Th.\@ 3.4]{konijnenbergmt}. 
By taking a one-sided ideal, one can see that this assumption is necessary for the one-sided cases. 
But the following result shows that the assumption $n \ge 3$ cannot be omitted either in the two-sided
case, see also \cite[Th.\@ 3.10]{konijnenbergmt}. 

\begin{proposition}
Let $K$ be a field such that $p := \chr K = 0$ or $p = \chr K \ge n$. Suppose that $K \ncong \F_p$
in the case where $p \in \{n,n+1\}$. Take $a = 1$ if $p \notin \{n,n+1\}$ and take 
$a \in K \setminus \F_p$ if $p \in \{n,n+1\}$. Then the subspace 
$$
\Ma := \big\{ M \in \Mat_n(K) ~\big|~ M_{n1} = M_{n2} = \cdots = M_{n(n-1)} = 0 
                                        = \tr M + a M_{nn}\big\}
$$
of $\Mat_n(K)$ is a two-sided Mathieu subspace of $\Mat_n(K)$.
\end{proposition}

\begin{proof}
On account of \cite[Lem.\@ 4.1]{MR2859886} and \cite[Cor.\@ 4.3]{MR2859886}, it suffices to
show that $\Ma$ has no nontrivial idempotent. Hence assume that $E \in \Ma$ is a nonzero idempotent.
Since $E_{n1} = E_{n2} = \cdots = E_{n(n-1)} = 0$, we see that the trailing principal minor matrix 
$E_{nn}$ of size $1$ of $E$ is an idempotent as well. 

By looking at Jordan normal forms, we see that the rank and the trace 
of an idempotent matrix over $K$ are equal within $K$. 
Hence $\tr E = \rk E \in \{1,2,\ldots,n\}$ and 
$E_{nn} \in \{0,1\}$. It follows that 
\begin{equation} \label{Eeq}
0 = \tr E + a E_{nn} \in \{1,2,\ldots, n\} + \{0,a\}
\end{equation}
within $K$. If $p = 0$ or $p \ge n+2$, then $a = 1$ and $n + a < n+2$, so
\eqref{Eeq} cannot be satisfied. Hence $p \in \{n,n+1\}$ and $a \notin \F_p$.
From \eqref{Eeq} and  $a \notin \F_p$, it follows that $E_{nn} = 0$, so 
$\tr E = \rk E < n$. This contradicts \eqref{Eeq}, because $p \ge n$.
\end{proof}

\noindent
If $K$ is closed under taking square root and $\chr K \ne 2$, then the above proposition 
gives all two-sided Mathieu subspaces of $\Mat_2(K)$ of (co)dimension two
which are not contained in $H$ up to linear conjugation, except that 
$a$ may be any element of $K$ such that $\Ma \nsubseteq H$ and the left hand side of 
\eqref{Eeq} is not contained in the right hand side of \eqref{Eeq}, i.e.\@
$$
a \ne 0 \qquad \mbox{and} \qquad 0 \notin \{1,2\} + \{0,a\} \qquad \mbox{in $K$}
$$
respectively, which comes down to that $a \notin \{-2,-1,0\}$ in $K$. 
This has been proved by Konijnenberg in \cite[Th.\@ 3.10]{konijnenbergmt}. 

Example 3.11 in \cite{konijnenbergmt} shows that the codimension $n$ case
seems quite difficult. Hence we shift our focus to subspaces of $\Mat_n(K)$ 
of codimension less than $n$ from now on. But let us first say something 
about subspaces of $H$.

\begin{lemma} \label{prelm}
Assume $\Ma$ is a subspace of $\Mat_n(K)$ such that $\tr M = 0$ for all $M \in \Ma$.
Then for
\begin{enumerate}[{\upshape (1)}]

\item $\chr K = 0$ or $\chr K \ge n+1$,

\item $\chr K = 0$ or $\chr K \ge n$ and $I_n \notin \Ma$,

\item Every element of $\rd(\Ma)$ is nilpotent,

\item $\Ma$ is a two-sided Mathieu subspace,

\end{enumerate}
we have {\upshape (1) $\Rightarrow$ (2) $\Rightarrow$ (3) $\Rightarrow$ (4)}.
\end{lemma}

\begin{proof}
If $\chr K = 0$ or $\chr K \ge n+1$, then $\tr I_n = n \ne 0$ in $K$. 
This gives (1) $\Rightarrow$ (2). Since (3) $\Rightarrow$ (4) follows directly 
from the definition of Mathieu subspace, (2) $\Rightarrow$ (3) remains to be proved.

So assume that $A \in \rd(\Ma)$. Then there exists an $N$ such that $A^m \in \Ma$ for all
$m \ge N$. Now let $B$ be the Jordan normal form of $A^N$ (or any other triangular matrix that is 
linearly conjugate to $A^N$). Then $\tr B^m = \tr A^{mN} = 0$
for all $m \ge 1$. Let $\beta$ be the diagonal of eigenvalues of $B$. Then $\sum_{i=1}^n \beta_i^m = 0$ 
for all $m$. 

Using the Newton identities on the eigenvalues of $B$, we get that the eigenvalue
polynomial of $B$ is of the form $t^n + (-1)^n \det B$ if $\chr K = 0$ or $\chr K \ge n$
(where $\det B = 0$ in case  $\chr K = 0$ or $\chr K \ge n+1$). From the Cayley-Hamilton theorem,
we deduce that $B^n + (-1)^n (\det B) I_n = 0$. It follows that $B^n$ and hence also $A^{nN}$
is a multiple of $I_n$. So either $I_n \in \Ma$ or $A^{nN} = 0$. This gives (2) 
$\Rightarrow$ (3).
\end{proof}

\noindent
Our main theorem, theorem \ref{main1} below , is that we indeed have $\Ma \subseteq H$ if the 
codimension is less than $n$, provided the base field $K$ is large enough. Sections 3 to 5
will be devoted to the highly technical proof of theorem \ref{main1}. But first, we will
give a rough sketch of this  proof in the next section.

\begin{theorem} \label{main1}
Let $K$ be a field.
Assume $\Ma$ is a proper Mathieu subspace of any type of $\Mat_n(K)$ of codimension less than 
$\min \{n, \#K\}$. Then $\tr M = 0$ for all $M \in \Ma$. In particular, every element of $\rd(\Ma)$ 
is nilpotent and $\Ma$ is a two-sided Mathieu subspace if $\chr K = 0$ or $\chr K \ge n$.
\end{theorem}

\noindent
Using theorem \ref{main1}, we can classify all Mathieu subspaces of $\Mat_n(K)$ of codimension less 
than $n$, under the assumption that $\chr K = 0$ or $\chr K \ge n$.

\begin{corollary}
Let $K$ be a field such that $\chr K = 0$ or $\chr K \ge n$.
Then for a proper $K$-subspace $\Ma$ of $\Mat_n(K)$ of codimension less than $n$, 
$\Ma$ is a Mathieu subspace of any arbitrary type of $\Mat_n(K)$, if and only if 
$\tr M = 0$ for all $M \in \Ma$, and either $\chr K \ne n$ or $I_n \notin \Ma$.
\end{corollary}

\begin{proof}
The `if'-part follows from lemma \ref{prelm}. To show the `only if'-part, suppose that 
$\Ma$ is a proper $K$-subspace of codimension less than $n$ of $\Mat_n(K)$. 
Since $\#K \ge \chr K \ge n$ if $\#K < \infty$, it follows that $\#K \ge n$ in any case and
that the codimension of $\Ma$ is less than
$\min \{n, \#K\}$. So $\tr M = 0$ for all $M \in \Ma$ on account of theorem \ref{main1}.
Since $I_n$ is not nilpotent and $I_n$ is contained in $\rd(\Ma)$
as soon as it is contained in $\Ma$, it additionally follows from theorem \ref{main1}
that $I_n \notin \Ma$, which completes the `only if'-part.
\end{proof}

\section{Sketch of the proof of theorem \ref{main1}}

Suppose that $\Ma$ is a subspace of codimension $c$ of $\Mat_n(K)$. Then the matrices $C \in \Mat_n(K)$ 
such that $\tr CM = 0$ for all $M \in \Ma$, which we call {\em constraints of $\Ma$}, form a subspace 
of dimension $c$ of $\Mat_n(K)$. Write $\Co$ for this space of constraints of $\Ma$.

Write $\rct(C)$ for the largest submatrix above the diagonal of $C \in \Mat_n(K)$ with $r$ rows.
So $\rct(C)$ has $n-r$ columns, corresponding to columns $r+1, \allowbreak r+2,\ldots, n$ of $C$.
Notice that $\tr CM$ is the sum of the entries of the Hadamard product of $C$ and the transpose
$M\tp$ of $M$. So the entries of $\rct(C)$ act as coefficients for the entries of $\rct(M\tp)$ and its
transpose, which we call $\rctt(M)$. So $\rctt(M)$ is the largest submatrix below the diagonal
of $M \in \Mat_n(K)$ with $r$ columns or $n-r$ rows.

The reason for using the formula $\tr CM = 0$ in the definition of constraint, 
instead of the sum of the entries of the Hadamard product, is that when we replace 
$\Ma$ by the isomorphic space $T^{-1} \Ma T$ for some $T \in \GL_n(K)$, 
the corresponding space of constraints $\Co$ gets replaced in a similar manner,
namely by $T^{-1} \Co T$.

In theorem \ref{rectang}, it is shown that $\Ma$ has idempotents of the forms 
$$
\left( \begin{array}{cc} I_r & \emptyset \\ * & \emptyset \end{array} \right)
\qquad \mbox{and} \qquad 
\left( \begin{array}{cc} \emptyset & \emptyset \\ * & I_{n-r} \end{array} \right)
$$
if $\rct$ is injective on $\Co$, after which theorem \ref{main1} is proved. The idea behind
theorem \ref{rectang} is more or less the following. We fix an arbitrary idempotent matrix $E$ of one of
both forms. Now for each nonzero $C \in \Co$, we want to have $\tr CE = 0$. Since
$\rct(C)$ is not the zero matrix, we can obtain $\tr CE = 0$ for some $C \in \Co$ by only changing the 
submatrix $\rctt(E)$ of $E$. The proof of theorem \ref{rectang} shows that this can be done for all
nonzero $C \in \Co$ simultaneously, so that $E$ can be changed into an idempotent of $\Ma$
by only changing the submatrix $\rctt(E)$ of $E$.

\medskip \noindent
The hard part of the proof of theorem \ref{main1} is the proof of theorem \ref{main2}.
This theorem claims that under certain conditions, among which $I_n \notin \Co$,
we can indeed obtain injectivity of $\rct$ on $\Co$ for some $r$ by way of linear 
conjugation. This conclusion leads to a contradiction
in the proof of theorem \ref{main1}, so that $I_n \notin \Co$ can be ruled out. 
So we get $I_n \in \Co$, which is equivalent to the main conclusion of theorem \ref{main1}:
$\tr M = 0$ for all $M \in \Ma$.

More precisely, the assertion of theorem \ref{main2} is the following. 
If we have $I_n \notin \Co$ besides certain conditions that are implied by those of theorem \ref{main1}, 
then after replacing $\Co$ by $T^{-1} \Co T$ for an appropriate $T \in \GL_n(K)$, 
there exists an $r$ with $1 \le r \le n-1$, such that $\rct(C)$ is not the zero matrix for all 
nonzero $C \in \Co$. We will even obtain a stronger conclusion: all nonzero entries of the rightmost 
nonzero column of any nonzero $C \in \Co$ are in $\rct(C)$. 

More formally, let $B' \in \Mat_n(\{0,1\})$ be the binary matrix, which
is $1$ on some spot, if and only if some element of $\Co$ has a rightmost nonzero entry at that spot.
Then we will obtain that all entries $1$ of $B'$ are in $\rct(B')$.
For that purpose, we apply a conjugation process on the space of
constraints, but not on $\Co$ directly. In order to get the conjugation process in the way we want,
we add the identity to $\Co$ by defining $\Co_n = \Co \oplus K I_n$, and apply the 
conjugation process on $\Co_n$.
 
We can easily reason out $I_n$ afterwards, because $I_n$ is not affected by conjugations.
Namely, if we define $B \in \Mat_n(\{0,1\})$ as the binary matrix, which
is $1$ on some spot, if and only if some element of $\Co_n$ has a rightmost nonzero entry at that spot,
then we will obtain that all entries $1$ of $B$ except $B_{nn} = 1$ are in $\rct(B)$. We have
$B_{nn} = 1$ because $I_n \in \Co_n$. With that, we have the only difference with $B'$, so that
$B' = B - e_n e_n\tp$, where $e_i$ is the $i$-th standard basis unit vector as a column vector.
In theorem \ref{main3}, we prove that we can obtain several properties
for $B$, which we discuss below. Under the conditions of theorem \ref{main1}, these properties 
imply that all entries $1$ of $B$ except $B_{nn}$ are in $\rct(B)$.

\medskip \noindent
Since we want the rightmost nonzero column of nonzero $C \in \Co$ to have some property, we define
$\Co_k$ as the space of $C \in \Co_n$ for which the rightmost nonzero column has index at most $k$, where
the index of the rightmost nonzero column of the zero matrix is $0$. So $\Co_0$ only contains the
zero matrix, and $\Co_n$ is correctly defined into itself.
Now we can define $B \in \Mat_n(K)$ alternatively by
$B_{ik} = 1$, if and only if $C_{ik} \ne 0$ for some $C \in \Co_k$. This is equivalent to
that $v_i \ne 0$ for some $v \in K^{\times n}$ which is $k$-th column of some matrix in $\Co_k$.

So the subspace formed by the $k$-th columns of matrices in $\Co_k$, 
which is isomorphic to $\Co_k / \Co_{k-1}$, plays a crucial role. 
We will denote this subspace of $\Co_k$ as $\Co_k e_k$, where
$e_k$ is the $k$-th standard basis unit vector. More generally, we define
$$
\Co_k v := \{Cv \mid C\in \Co_k\} \qquad (v \in K^{\times n})
$$
The dimension of 
$\Co_k v$ does not exceed $n$ and neither exceeds $\dim \Co_k$. If $K$ is infinite, then we take
for $d_k$ the maximum dimension that a space of the form $\Co_k v$ can have. In general,
we first replace $K$ by an infinite extension field $L$ of $K$, which is done by 
taking the tensor product over $K$ of $L$ and $\Co_k$, and next take $v \in L^{\times n}$.

So $d_k$ is the maximum dimension that a space of the form $(L \otimes_K \Co_k) \cdot v$ can have,
where $v \in L^{\times n}$. Our choice of $d_k$ is highly ambiguous, but lemma \ref{dimL} shows that
$d_k$ is still uniquely determined. For the actual definition of $d_k$, which is not ambiguous,
we take $L = K(x)$ and $v = x$, where $x = (x_1, x_2, \ldots, x_n)$. In any case, we have
$d_k \ge \max\{ \dim \Co_k v \mid v \in K^{\times n}\} \ge \dim \Co_k e_k$.

\medskip \noindent
In the proof of theorem \ref{main3}, we arrange the required properties for $B$ by way of linear 
conjugation of $\Co_n$. In order to do that, we additionally ensure that we get 
$d_k = \dim \Co_k e_k$ for each $k$. We achieve this by doing the following for $k = n,n-1,\ldots,1$,
in that order. We first 
choose a $v \in K^{\times n}$ such that $d_k = \dim \Co_k v$. If $\#K \ge d_k$, then such a $v$ indeed exists,
and if $\#K > d_k$, then we can additionally choose $v$ such that $v_k = 1$ (see lemma \ref{vcol}).
Next, we choose $T \in \GL_n(K)$ such that the $k$-th column $T e_k$ of $T$ equals $v$, and
replace $\Co_n$ by $T^{-1} \Co_n T$. 

By choosing $T$ properly (namely equal to the identity matrix at the right of the $k$-th column), 
$\Co_k$ gets replaced by $T^{-1} \Co_k T$, so that $\Co_k e_k$ gets replaced by 
$T^{-1} \Co_k T e_k = T^{-1} \Co_k v$, which is isomorphic to $\Co_k v$ (see (ii) of lemma \ref{Tlem}). 
This is how $d_k = \dim \Co_k e_k$ is obtained, but what we ignore here is the problem, 
that $d_j = \dim \Co_j e_j$ for $j > k$ and possibly some other properties
which $B$ already satisfies, should not be affected. Such preservation problems, 
which we will mostly ignore in this section, makes the proof of theorem \ref{main3} 
highly technical in nature. 

One of these preservation problems can be solved if we can take $v_k = 1$, because in that case,
we can choose $T$ equal to the identity matrix outside its $k$-th column $v$. If we additionally take 
$v_{k+1} = v_{k+2} = \cdots = v_n = 0$, which is also possible, then the $j$-th column of $B$ will be 
preserved for all $j > k$, provided this $j$-th column of $B$ is decreasing above the diagonal (see 
the proof of (ii) of theorem \ref{main3}). 

Once we have $d_k = \dim \Co_k e_k$ in theorem \ref{main3}, we additionally have that $B$ is increasing
in every row, provided $\#K \ge d_n$ (this is shown in (ii) of lemma \ref{zeroL}, where lemma \ref{dimL}
is used to obtain the condition of lemma \ref{zeroL}).

Write $b_j$ for the number of ones in column $j$ of $B$.
Another property to arrange is that $b_k = \dim \Co_k e_k$ as well as $d_k = \dim \Co_k e_k$, and lemma 
\ref{small} tells us that for this purpose, $\Co_k e_k$ should be spanned by standard basis unit vectors.
We do this by taking $L \in \GL_n(K)$ lower triangular, such that $L \Co_k e_k$ is spanned by 
standard basis unit vectors. Since $L$ is lower triangular and invertible, we can prove that $\Co_k e_k$ is 
replaced by $L \Co_k e_k$ if $\Co_n$ is replaced by $L \Co_n L^{-1}$ (see (ii) of lemma \ref{Tlem}). 
So $\Co_k e_k$ will be spanned by standard basis unit vectors after this replacement.

If $\#K > \min\{d_{n-1},n-1\}$, then we must additionally obtain that every column of $B$ is increasing
above the diagonal. If $\Co_k e_k$ is spanned by standard basis unit vectors, then there exists
a permutation $P$ such that $P \Co_k e_k$ is spanned by the first $\dim \Co_k e_k$ standard basis
unit vectors. But if we replace $\Co_n$ by $P \Co_n P^{-1}$, then property $d_j = \dim \Co_j e_j$
could be affected for some $j > k$, as well as several properties that $B$ satisfies. 

For this reason, we take $P$ such that only the first $k-1$ coordinates of $\Co_k e_k$ are permuted.
These coordinates correspond to the part above the diagonal of the $k$-th column of $B$.
If we replace $\Co_n$ by $P \Co_n P^{-1}$, then $\Co_k e_k$ will be replaced by $P \Co_k e_k$
(see (ii) of lemma \ref{Tlem}). In order to preserve properties of $B$, $P$ and also $L$ only act on 
coordinates $i$ such that $B_{ij} = 1$ for all $j > k$, so that $v \mapsto Lv$ and $v \mapsto Pv$
are isomorphisms of $\Co_j e_j$ for all $j > k$ (as far as the respective $\Co_j e_j$ are generated by 
standard basis unit vectors, but that is inductively arranged).

\section{A criterium for having idempotents}

Let $K$ be a field and $\Ma$ be a subspace of $\Mat_n(K)$. Let $\Co$ be the subspace of matrices 
$C \in \Mat_n(K)$ such that $\tr CM = 0$ for all $M \in \Ma$. Let $x = (x_1,x_2,\ldots, x_n)$
be $n$ indeterminates. Write $\rct(M)$ for the submatrix consisting of the first $r$ rows
and the rightmost $n-r$ columns of $M$ for all $M \in \Mat_n(K)$.

\begin{theorem} \label{rectang}
Suppose that $1 \le r \le n-1$. 
\begin{enumerate}[{\upshape (i)}]

\item If for all $C \in \Co$ such that $\rct(C)$ is the zero matrix,
the leading principal minor matrix of size $r$ of $C$ has trace zero, then $\Ma$ contains an 
idempotent of rank $r$ of the form
\begin{equation} \label{idp}
\left( \begin{array}{cc} I_r & \emptyset \\ * & \emptyset \end{array} \right)
\end{equation}

\item If for all $C \in \Co$ such that $\rct(C)$ is the zero matrix,
the trailing principal minor matrix of size $n-r$ of $C$ has trace zero, then $\Ma$ contains an 
idempotent of rank $n-r$ of the form
\begin{equation} \label{idp'} 
\left( \begin{array}{cc} \emptyset & \emptyset \\ * & I_{n-r} \end{array} \right) \tag{\ref{idp}$'$}
\end{equation}

\end{enumerate}
More precisely, the dimensions of the affine spaces of idempotents in $\Ma$ of the forms 
\eqref{idp} and \eqref{idp'} respectively, are both equal to that of 
$$
\Na := \left\{ M \in \Ma \,\left|\, M = \left( \begin{array}{cc} \emptyset & \emptyset \\ 
\tilde{M} & \emptyset \end{array} \right) \mbox{ for some } \tilde{M} \in \Mat_{n-r,r}(K) 
\right.\right\}
$$
\end{theorem}

\begin{proof}
Since (ii) is similar to (i) (or take the transpose and conjugate with the reversing permutation to 
reduce to (i)), we only prove (i). Notice that any matrix of the form \eqref{idp} is an idempotent of 
rank $r$, and that $0 \le \dim_K \Na \le (n-r)r$ because $0 \in \Na$. Take $m$ such that 
$\dim_K \Na = (n-r)r - m$. Then there are constraints $C_1, C_2, \ldots, C_m \in \Co$ such that
$$
\Na = \left\{\left. M = \left( \begin{array}{cc} \emptyset & \emptyset \\ 
\tilde{M} & \emptyset \end{array} \right) \mbox{ for some } \tilde{M} \in \Mat_{n-r,r}(K) 
\,\right|\, \tr C_i M = 0 \mbox{ for all } i\right\}
$$
$\tilde{M}$ and its transpose $\tilde{M}\tp$ are submatrices of $M$ and its transpose $M\tp$ 
respectively, and we have $\tilde{M}\tp = \rct(M\tp)$.
By definition of $\Na$, for any constraint $C' \in \Co$, $\rct(C')$ is contained in the span of
the corresponding submatrices of $C_1, C_2, \ldots, C_m \in \Co$. Hence we can write each
$C' \in \Co$ as
\begin{equation} \label{Czero}
C' = \lambda_1 C_1 + \lambda_2 C_2 + \cdots + \lambda_m C_m + C^{*}
\end{equation}
such that $\rct(C^{*})$ is the zero matrix. By assumption, the leading principal minor 
matrix of size $r$ of $C^{*}$ has trace zero. Hence we have
\begin{equation} \label{CE}
\tr C^{*} E = 0
\end{equation}
for all $E$ of the form \eqref{idp}.

Since there are $m$ independent constraints on essentially $(n-r)r + 1$ coordinates, 
the dimension of the space
$$
\left\{\left. M = \left( \begin{array}{cc} \lambda I_r & \emptyset \\ 
\tilde{M} & \emptyset \end{array} \right) \,\right|\, \lambda \in K, \tilde{M} \in \Mat_{n-r,r}(K) 
\mbox{ and } \tr C_i M = 0 \mbox{ for all } i\right\}
$$
is $(n-r)r + 1 - m$, which is one larger than that of its subspace $\Na$. Hence the dimension of
its affine subspace 
$$
\Ea := \left\{\left. M = \left( \begin{array}{cc} I_r & \emptyset \\ 
\tilde{M} & \emptyset \end{array} \right) \,\right|\, \tilde{M} \in \Mat_{n-r,r}(K) 
\mbox{ and } \tr C_i M = 0 \mbox{ for all } i\right\}
$$
which contains all idempotents of the form \eqref{idp} in $\Ma$, is $(n-r)r - m$, just as the dimension of 
$\Na$. 

Now it remains to show that $\Ea$ does not contain any idempotent outside $\Ma$. 
For that purpose, let $E \in \Ea$ and suppose that there exist a $C' \in \Co$ such that $\tr C'E \ne 0$. 
By \eqref{Czero} and by definition of $\Ea$, 
there exists a $C^{*} \in \Co$ such that $\tr C^{*}E \ne 0$ and $\rct(C^{*})$ is the zero matrix. 
This contradicts \eqref{CE}, so a $C'$ as above does not exist and we have $\Ea \subseteq \Ma$.
Hence $\Ea$ is the affine subspace of idempotents of the form \eqref{idp} in $\Ma$.
\end{proof}

\begin{corollary} \label{rectangcor}
Assume $I_n \notin \Co$ and suppose that for some $r$ with $1 \le r \le n-1$ we have the following:
all $C \in \Co \oplus K I_n$, such that $\rct(C)$ is the zero matrix, are dependent of $I_n$. 
Then $\Ma$ contains an idempotent of rank $r$ and another one of rank $n-r$, such that the sum of both 
idempotents is unipotent.

Furthermore, if $\Ma$ is a Mathieu subspace of any type, then $\Ma = \Mat_n(K)$ and $\Co = 0$.
\end{corollary}

\begin{proof}
Since $I_n \notin \Co$, we see that all $C \in \Co$, such that $\rct(C)$ is the zero matrix, are 
entirely zero by assumption. By (i) of theorem \ref{rectang}, $\Ma$ contains an idempotent of the form
$$
E := \left( \begin{array}{cc} I_r & \emptyset \\ * & \emptyset \end{array} \right)
$$
By (ii) of theorem \ref{rectang}, $\Ma$ contains another idempotent of the form
$$
E' := \left( \begin{array}{cc} \emptyset & \emptyset \\ * & I_{n-r} \end{array} \right)
$$
Notice that $E + E'$ is unipotent and hence invertible. If $\Ma$ is a left Mathieu subspace and 
$A \in \Mat_n(K)$, then 
$$
A = A I_n = A (E + E')^{-1} E + A (E + E')^{-1} E' \in \Ma
$$
because $E^m = E \in \Ma$ and $(E')^m = E' \in \Ma$ for all $m \ge 1$. Thus
$\Ma = \Mat_n(K)$ and $\Co = 0$ in the case where $\Ma$ is a left Mathieu subspace. The case 
where $\Ma$ is a right Mathieu subspace is similar. 
\end{proof}

\noindent
Write $x$ be the column vector $(x_1, x_2, \ldots, x_n)$. We will show that theorem \ref{main2} below
implies theorem \ref{main1}.

\begin{theorem} \label{main2}
Suppose that $I_n \notin \Co$ and $0 < \dim_K \Co < n$. Let $\Co_n = \Co \oplus K I_n$ and suppose that
\begin{align*}
\#K \ge r+1 &:= \dim_{K(x)} \big((K(x) \otimes_K \Co_n) \cdot x \big) \\ 
 &\hphantom{:}= \dim_{K(x)} \sum_{C \in \Co_n} K(x) \cdot C \cdot x
\end{align*}
Then we can obtain corollary \ref{rectangcor} (with a corresponding $r$) by way of linear 
conjugation (replacing $\Ma$ by $T^{-1} \Ma T$ and $\Co$ by $T^{-1} \Co T$ for some $T \in \GL_n(K)$).
\end{theorem}

\begin{proof}[Proof of theorem \ref{main1}.]
The primary result to show is, that $\tr M = 0$ for all $M \in \Ma$. This is equivalent to 
$I_n \in \Co$, so suppose that $I_n \notin \Co$. Let $\Co_n = \Co \oplus K I_n$. 
By assumption, $\dim_K \Co < \min \{n, \#K\}$. Hence $\dim_K \Co < n$ and
$$ 
\dim_{K(x)} \sum_{C \in \Co_n} K(x) \cdot C \cdot x \le 
\dim_{K(x)} \sum_{C \in \Co_n} K(x) \cdot C = \dim_K \Co_n \le \#K
$$
Now theorem \ref{main2} above gives a contradiction, so $I_n \in \Co$ and hence
$\tr M = 0$ for all $M \in \Ma$. 

Since $\Ma$ is proper by assumption, $I_n \notin \Ma$. Hence the secondary results
follow from (2) $\Rightarrow$ (3) $\Rightarrow$ (4) of lemma \ref{prelm}.
\end{proof}

\section{A binary matrix about a filtration on the constraint space}

Write $e_i$ for the $i$-th standard basis unit vector as a column vactor. 
Let $\Co_n$ be a $K$-subspace of $\Mat_n(K)$ and define
$$
\Co_k := \{ C \in \Co_n \mid C e_{k+1} = C e_{k+2} = \cdots = C e_n = 0 \}
$$
Then $0 = \Co_0 \subseteq \Co_1 \subseteq \Co_2 \subseteq \cdots \subseteq \Co_n$ is a filtration
in the sense that we can take quotients $\Co_j / \Co_{j-1}$, which are isomorphic to
$\Co_j e_j$, where $\Co_j v := \{Cv \mid C \in \Co_j\}$.
Define the binary matrix $B \in \Mat_n(\{0,1\})$ by
$$
B_{ij} := \dim_K e_i\tp \Co_j e_j
$$
for all $i,j$, where 
$$
e_i\tp \Co_j v := \{e_i\tp Cv \mid C \in \Co_j\} = \{(Cv)_i \mid C \in \Co_j\}
$$
Write $b_j$ for the number of ones in column $j$ of $B$.

Theorem \ref{main3} below can be formulated in terms of the binary matrix $B$.
We will show that it implies theorem \ref{main2} (and hence also theorem \ref{main1}).
The next section will be devoted to the proof of theorem \ref{main3}.

\begin{theorem} \label{main3}
Suppose that $\#K \ge r+1$, where $r+1$ is as defined in theorem \ref{main2}. 
By way of linear conjugation, we can obtain the following.
\begin{enumerate}[{\upshape (i)}]
 
\item $b_j = \dim_K \Co_j e_j$ for all $j$, $b_n = r+1$, and $B$ is increasing in every 
row, i.e.\@ $B_{ij} = 0$ implies $B_{i(j-1)} = 0$ for every $i,j$ such that $j > 1$.

\item If $\#K > \min\{b_{n-1},n-1\}$, then $B$ is decreasing above the diagonal 
in every column, i.e.\@ $B_{ij} = 0$ implies $B_{(i+1)j} = 0$ for every $i,j$ such that $i + 1 < j$.

\item If $I_n \in \Co_n$, then $b_n > \min\{b_{n-1},n-1\}$ and $B_{(n-1)n} \ge B_{n(n-1)}$.

\end{enumerate}
\end{theorem}

\begin{proof}[Proof of theorem \ref{main2}.]
On account of theorem \ref{main3}, we can apply a linear conjugation on $\Co_n$ such that the
assertions of theorem \ref{main3} are satisfied. By (i), we have $\#K \ge r+1 = b_n$ and by (iii), 
we have $b_n > \min\{b_{n-1},n-1\}$. Hence the condition $\#K > \min\{b_{n-1},n-1\}$ 
in (ii) is fulfilled.
\begin{enumerate}[(i)] 

\item
\emph{We first show that the first $r$ columns of $B$ are zero.}
For that purpose, take $k$ minimal such that 
$b_k \ge 1$. On account of (i) of theorem \ref{main3}, we even have $b_j \ge 1$ for all $j \ge k$.
Since $\Co_j / \Co_{j-1}$ is isomorphic to $\Co_j e_j$ for all $j$, it follows from
(i) of theorem \ref{main3} that $b_j = \dim_K \Co_j e_j = \dim_K \Co_j / \Co_{j-1}$ for all $j$, and
\begin{alignat*}{7}
n &\ge \dim_K \Co + 1 = \dim_K \Co_n \hspace*{-50cm} \\
&= \dim_K \Co_1 / \Co_0 + \dim_K \Co_2 / \Co_1 + \cdots + \dim_K \Co_n / \Co_{n-1} \hspace*{-50cm}\\
&=   b_1 &&{}+ \cdots &&{}+ b_{k-1} &&{}+ b_k &&{}+ \cdots &&{}+ b_{n-1} &&{}+ b_n \\
&\ge 0   &&{}+ \cdots &&{}+ 0       &&{}+ 1   &&{}+ \cdots &&{}+ 1       &&{}+ (r + 1) \hspace{11mm} \\ 
&= n - k + r + 1 \hspace*{-50cm}
\end{alignat*}
So $k \ge r+1$ and indeed $b_1 = b_2 = \cdots = b_r = 0$.

\item
\emph{We next show that $B_{nn} = 1$ is the only nonzero entry in the last $n - r$
rows of $B$.} At first, $B_{nn} = 1$ follows directly from $I_n \in \Co_n$.
If $r = n-1$, then $B_{nj} = 0$ for all $j \le n-1$ because of (i) above, which gives the claimed
result. Hence assume that $r < n-1$.
Since $b_n = r+1$, it follows from (ii) of theorem \ref{main3} that $B_{(r+1)n} = B_{(r+2)n} = 
\cdots = B_{(n-1)n} = 0$. In particular $B_{(n-1)n} = 0$, and (iii) of theorem \ref{main3} subsequently 
gives $B_{n(n-1)} = 0$. By (i) of theorem \ref{main3}, every row of $B$ is increasing.
Hence every entry in the last $n - r$ rows of $B$ that has not been mentioned yet is zero as well. 

\end{enumerate} 
Take $C \in \Co_n$ such that $\rct(C)$ is the zero matrix. We must show that $C = \lambda I_n$ for some 
$\lambda \in K$. Take $\lambda \in K$ such that the lower right corner entry of $C' := C - \lambda I_n$ 
is zero. Notice that $\rct(C') = \rct(C)$. We must show that $C' = 0$.

So assume that $C' \ne 0$. Take $k \le n$ maximal such that $C'_{ik} \ne 0$ for some $i \le n$.
Then $C' \in \Co_k$ and the $i$-th coordinate of $C' e_k$ is nonzero, so $B_{ik} = 1$. 
On account of (i), we have $k \ge r+1$, and by the fact that $\rct(C')$ is the zero matrix,  
$i \ge r+1$ as well. By (ii), we even have $i = k = n$, so $C'_{nn} \ne 0$. Contradiction.
\end{proof}

\noindent
The following lemma is not very hard, but it is used several times.

\begin{lemma} \label{small}
For all $j$, we have
$$
b_j \ge \dim_K \Co_j e_j
$$
and equality holds, if and only if $B_{ij} e_i \in \Co_j e_j$ for all $i$,
if and only if $\Co_j e_j$ is the linear span of standard basis unit vectors.
\end{lemma}

\begin{proof}
Notice that $\Co_j e_j$ is the linear span of standard basis unit vectors, if and only if for all $i$
such that $\Co_j e_j$ is nontrivial at the $i$-th coordinate, we have $e_i \in \Co_j e_j$.
This is equivalent to that $B_{ij} e_i \in \Co_j e_j$ for all $i$.

Let $U$ be the linear span of the standard basis unit vectors $e_i$ for which 
$e_i \Co_j e_j \ne \{0\}$.
Then $U$ is a space of dimension $b_j$ which contains $\Co_j e_j$. So $b_j \ge \dim_K \Co_j e_j$, and if 
$b_j = \dim_K \Co_j e_j$, then $\Co_j e_j = U$ is the linear span of standard basis unit vectors.

If $b_j > \dim_K \Co_j e_j$, then there must be a standard basis unit vector of $U$ that is not contained 
in $\Co_j e_j$, while the corresponding coordinate projection of $\Co_j e_j$ is nontrivial. So 
$\Co_j e_j$ is not the linear span of standard basis unit vectors if $b_j > \dim_K \Co_j e_j$.
\end{proof}

\noindent
In order to prove theorem \ref{main3}, we will use the following lemma. The assertion that
$B_{ij} = 0$ implies $B_{i(j-1)} = 0$ can be found in the conclusion of (ii). 
Taking $k = n-1$ in the conclusion of (iii) 
gives $B_{(n-1)n} \ge B_{n(n-1)}$, which is another assertion of theorem \ref{main3}.

\begin{lemma} \label{zeroL}
Suppose that $\dim_K \Co_j e_j \ge \dim_{K(x_j)} \big((K(x_j) \otimes_K \Co_j) \cdot (e_k + x_j e_j)
\big)$. Then we have the following.
\begin{enumerate}[{\upshape (i)}]

\item $\dim_K \Co_j e_j = \dim_{K(x_j)} \big((K(x_j) \otimes_K \Co_j) \cdot  (e_k + x_j e_j)\big)$.

\item If $B_{ij} = 0$ for some $i$, then $e_i\tp \Co_{j-1} e_k = \{0\}$ as well. 

In particular, $B_{ij} = 0$ implies $B_{i(j-1)} = 0$ if $k = j-1$.

\item If $B_{ij} = 0$ for some $i$ and there exists a $C' \in \Co_j$ such 
that $e_i\tp C' e_k \ne 0$, then $C' e_j \notin \Co_{j-1} e_k$. 

In particular, we have $B_{kn} \ge B_{nk}$ if $j = n$, $I_n \in \Co_n$ and 
$b_k = \dim_K \Co_k e_k$.

\end{enumerate}
\end{lemma}

\begin{listproof}
\begin{enumerate}[(i)]

\item Let  $d := \dim_K \Co_j e_j$. Then we can find $C_1, C_2, \ldots, C_d \in \Co_j$
such that $\Co_j e_j = K C_1 e_j \oplus K C_2 e_j \oplus \cdots \oplus K C_d e_j$. Hence the 
$n \times d$ matrix
$$
\Big( C_1 x_j e_j \Big| C_2 x_j e_j \Big| \cdots \Big| C_d x_j e_j \Big)
$$
has a minor of size $d$ which has degree $d$. 
The corresponding minor of the $n \times d$ matrix
$$
\Big( C_1 (e_k + x_j e_j) \Big| C_2 (e_k + x_j e_j) \Big| \cdots \Big| C_d (e_k + x_j e_j) \Big)
$$ 
has degree $d$ as well,
so $\dim_K \Co_j e_j \le \dim_{K(x_j)} \Co_j (e_k + x_j e_j)$, and (i) follows by assumption.

\item By taking $k = j - 1$, the last claim follows from the first claim. To prove the first
claim, suppose that $i \le n$ and that there exists a $C_{d+1} \in \Co_{j-1}$ such that 
$e_i\tp C_{d+1} e_k \ne 0$. Then $C_{d+1} e_j = 0$, so we have 
$C_{d+1} (e_k + x_j e_j) \in K^{\times n}$ and $e_i\tp C_{d+1} (e_k + x_j e_j) \in K^{*}$.
Suppose additionally that $B_{ij} = 0$. Then $e_i\tp \Co_j e_j = \{0\}$, so the $i$-th
rows of the matrices of size $n \times d$ in the proof of (i) are constant. It follows 
that the minors of these matrices in the proof of (i) do not use row $i$.

By expansion along the $i$-th row or the $(d+1)$-th column, which are both constant, 
we see that the  $n \times (d+1)$ matrix
$$
\Big( C_1 (e_k + x_j e_j) \Big| C_2 (e_k + x_j e_j) \Big| \cdots \Big| C_d (e_k + x_j e_j)
\Big| C_{d+1} (e_k + x_j e_j) \Big)
$$ 
has a minor of size $d + 1$ which has degree $d$, 
namely the minor of size $d$ in the proof of (i), extended with 
row $i$ and column $d + 1$. This contradicts $\dim_K \Co_j e_j \ge \dim_{K(x_j)} 
\big((K(x_j) \otimes_K \Co_j) \cdot (e_k + x_j e_j) \big)$.

\item We first show that the first claim implies the last claim. 
Take $i = k$, $j = n$ and $C' = I_n$ in the first claim.
Assuming the first claim, we see that $B_{kn} = 0$ and $I_n \in \Co_n$
together imply $e_n = I_n e_n \notin \Co_{n-1} e_k$. Now suppose that 
$B_{kn} < B_{nk}$ and $I_n \in \Co_n$. Then $k \le n-1$ and $B_{nk} = 1$, so that
$B_{nk} e_n = e_n \notin \Co_{n-1} e_k \supseteq \Co_k e_k$. From lemma \ref{small},
we deduce that $b_k > \dim_K \Co_k e_k$. This gives the last claim.

To prove the first claim, suppose that $B_{ij} = 0$ and there exists a $C' \in \Co_j$ 
such that $e_i\tp C' e_k \ne 0$. By (ii), we have $C' \notin \Co_{j-1}$, thus we may assume that $C_d = C'$ 
in the proof of (i). Just as in the proof of (ii), we can see that the minors in the proof of (i) 
does not use row $i$, because that row is constant with respect to $x_j$.

Suppose additionally that $C' e_j \in \Co_{j-1} e_k$. 
Then there exists a $C_{d+1} \in \Co_{j-1}$ such that $C_{d+1} e_k = C' e_j = C_d e_j$. 
Since $x_j C_{d+1} e_k = x_j C_d e_j$ and $x_j^2 C_{d+1} e_j \in x_j^2 \Co_{j-1} e_j = 0$, 
it follows that
\begin{align*} 
(x_j C_{d+1} - C_d) (e_k + x_j e_j) &= -C_d e_k \in K^{\times n} \\
\intertext{and} 
e_i\tp (x_j C_{d+1} - C_d) (e_k + x_j e_j) &= -e_i\tp C_d e_k = -e_i\tp C' e_k \in K^{*}
\end{align*}

By expansion along the $i$-th row or the $(d+1)$-th column, which are both constant, 
we see that the $n \times (d+1)$ matrix
$$\
\Big( C_1 (e_k + x_j e_j) \Big| C_2 (e_k + x_j e_j) \Big| \cdots \Big| C_d (e_k + x_j e_j)
\Big| (x_j C_{d+1} - C_d) (e_k + x_j e_j) \Big)
$$ 
has a minor of size $d + 1$ which has degree $d$, 
namely the minor of size $d$ in the proof of (i), 
extended with row $i$ and column $d + 1$. 
This contradicts $\dim_K \Co_j e_j \ge \dim_{K(x_j)} 
\big((K(x_j) \otimes_K \Co_j) \cdot (e_k + x_j e_j) \big)$.
\qedhere

\end{enumerate}
\end{listproof}

\section{Proof of theorem \ref{main3}}

The following two lemmas are not really necessary for the proof if the base field $K$ is infinite.

\begin{lemma} \label{vanish}
Let $K$ be a field and $f \in K[x] = K[x_1,x_2,\ldots,x_n]$ such that $\deg f \le d$. Suppose that
$S \subseteq K$ such that $f$ vanishes on $S^{\times n}$. Then $f = 0$ in the following cases.
\begin{enumerate}[{\upshape (i)}]

\item $\#S > d$,

\item $f$ is homogeneous, $0 \in S$ and $\#S \ge \max\{d,2\}$.

\end{enumerate}
\end{lemma}
 
\begin{proof}
By replacing $f(x)$ by $f(x-s)$ for some $s \in S$, we may assume that $0 \in S$ in (i) as well. 
Let $\tilde{S} = S \setminus \{0\}$.
\begin{enumerate}[(i)]

\item We can write
$$
f(x) = f(x_1,x_2,\ldots,x_{n-1},0) + x_n g(x)
$$
Notice that $f(x)$ and hence also $f(x_1,x_2,\ldots,x_{n-1},0)$ vanishes at 
$S^{\times (n-1)} \times \{0\}$. 
By induction on $n$, we deduce that $f(x_1,x_2,\ldots,x_{n-1},0) = 0$, so $x_n g$ vanishes at 
$S^{\times n}$. 
Since $x_n$ does not vanish anywhere at $\tilde{S}^{\times n}$, we conclude that $g$ vanishes at 
$\tilde{S}^{\times n}$. By induction on $d$, $g = 0$, so $f = 0$ as well.

\item If $x_n^2 \mid f$, then we can apply (i) on $x_n^{-1} f$ instead of $f$, to obtain $f = 0$. 
The case $\deg f \le 1$ follows from (i) as well. So assume that $\deg f \ge 2$ and $x_n^2 \nmid f$.
Then $n \ge 2$. Take $g(x)$ as in (i). Just as in (i), $f(x_1,x_2,\ldots,x_{n-1},0) = 0$ follows by induction
and $x_n g$ vanishes at $S^{\times n}$. Write
$$
x_n g(x) = x_n g(x_1,x_2\ldots,x_{n-2},0,x_n) + x_{n-1} x_n h(x)
$$
Notice that $x_n g(x)$ and hence also $x_n g(x_1,x_2,\ldots,x_{n-2},0,x_n)$ vanishes at 
$S^{\times (n-2)} \times \{0\} \times S$.
By induction on the number of variables, we deduce that $x_n g(x_1,x_2,\ldots,x_{n-2},0,x_n) = 0$. 
Hence $x_{n-1} x_n h(x)$ vanishes at $S^{\times n}$. Since $x_{n-1} x_n$ does not vanish anywhere at 
$\tilde{S}^{\times n}$, we conclude that $h$ vanishes at $\tilde{S}^{\times n}$. 
On account of $\#\tilde{S} \ge d-1 > d-2 \ge \deg h$, $h = 0$ follows from (i).
So $f = 0$ once again. \qedhere

\end{enumerate}
\end{proof}

\noindent
Suppose that $K$ has a $(q-1)$-th root of unity, e.g.\@ $K = \F_q$.
The polynomials $x_1^{q-1} - 1$ and $x_1^q - x_1$ show that $\#S > d$ 
is necessary in (i). The polynomials $x_1^{q-1} - x_2^{q-1}$ and $x_1^qx_2 - x_1x_2^q$ 
show that $0 \in S$ and $\#S \ge d$ respectively are necessary in (ii).

Notice that $1\frac12$ lies between the degrees of the leading and the trailing term of
$x_1^q - x_1$. Since $\deg (x_1^q - x_1) \ge \#\F_q$, this is no coincidence, because 
the homogeneity condition in (ii) can be replaced by that $1\frac12$ is not contained in the interval 
that envelops the term degrees (and the proof of (ii) still applies).

\begin{lemma} \label{dimL}
Let $L/K$ be a field extension (possibly trivial) and let $\Ve$ be a subspace of $\Mat_{m,n}(K)$.
Define
\begin{align*}
d &:= \dim_{K(x)} \big((K(x) \otimes_K \Ve) \cdot x\big) \\
  &\hphantom{:}= \dim_{K(x)} \sum_{V \in \Ve} K(x) \cdot V \cdot x
\end{align*}
Then we have the following.
\begin{enumerate}[{\upshape (i)}]

\item For all $v \in L^{\times n}$, we have $\dim_L\big((L \otimes_K \Ve) \cdot v\big) \le d$.

\item If $\# L \ge d$, then there exists a vector $v \in L^{\times n}$ such that 
$\dim_L\big((L \otimes_K \Ve) \cdot v\big) = d$.

\item If $\# L > d$, then for each $k \le n$, there exists a vector $v \in L^{\times n}$ such that 
$\dim_L\big((L \otimes_K \Ve) \cdot v\big) = d$ and $v_k = 1$.

\end{enumerate}
\end{lemma}

\begin{proof}
Let $D := \dim_K \Ve$ and take a basis $V_1, V_2, \ldots, V_D$ of $\Ve$. Since 
$D = \dim_{K(x)} (K(x) \otimes_K \Ve)$, $V_1, V_2, \ldots, V_D$ is also a basis of $K(x) \otimes_K \Ve$.
Hence $V_1 x, \allowbreak V_2 x, \allowbreak\ldots, V_D x$ is a spanning set of $(K(x) \otimes_K \Ve) \cdot x$. 
After an appropriate renumbering of the $V_i$'s, we have that $V_1 x, V_2 x, \allowbreak \ldots, V_d x$
is a basis of $(K(x) \otimes_K \Ve) \cdot x$. 
\begin{enumerate}[(i)]

\item
Take any $v \in L^{\times n}$. Notice that $V_1, V_2, \ldots, V_D$ is also a basis of $K(v) \otimes_K \Ve$.
Hence $V_1 v, V_2 v, \ldots, V_D v$ is a spanning set of $(K(v) \otimes_K \Ve) \cdot v$. 
Suppose that we have a subset $\{W_1, W_2, \ldots, W_{d+1}\}$ of $\{V_1, V_2, \ldots, V_D\}$
such that $W_1 v, W_2 v, \ldots, W_{d+1} v$ are independent over $K(v)$. Then the matrix with 
columns $W_1 v, W_2 v, \ldots, W_{d+1} v$ has a minor of size $d+1$ that does not vanish.
The corresponding minor of the matrix with columns $W_1 x, W_2 x, \ldots, W_{d+1} x$ 
does not vanish either, so $W_1 x, W_2 x, \ldots, W_{d+1} x$ are independent over $K(x)$.
This contradicts the definition of $d$, so if we reduce $V_1 v, V_2 v, \ldots, V_D v$ to a basis, we 
get $d' \le d$ vectors $W_1 v, W_2 v, \ldots, W_{d'} v$.

\item
If $d = 0$, then we can take $v$ arbitrary on account of (i), so assume that $d \ge 1$.
The matrix with columns $V_1 x, V_2 x, \ldots, V_d x$ contains a minor $h(x) \ne 0$ 
of size $d$ which has degree $d$ and is homogeneous. 
Suppose that $\# L \ge d$. By (ii) of lemma \ref{vanish}, there exists a 
vector $v \in L^{\times n}$ such that $h(x)$ does not vanish at $v$. Hence the matrix with columns
$V_1 v, V_2 v, \ldots, V_d v$ has a minor of size $d$ that does not vanish either.
This gives (ii).

\item
If $d = 0$, then we can take $v = (1,1\ldots,1)$ on account of (i), so assume that $d \ge 1$.
Suppose that $\# L > d$ and take any $k \le n$. Take $h(x)$ as in the proof of (ii).
By (ii) of lemma \ref{vanish}, there exists a vector $v \in L^{\times n}$ such that $x_k h(x)$ does not 
vanish at $v$. Hence we can deduce the conclusion of (ii) once again. Since we have $v_k \ne 0$
in addition, we can obtain $v_k = 1$ by dividing $v$ by $v_k$, because $h$ is homogeneous.
\qedhere

\end{enumerate}
\end{proof}

\noindent
From now on in this section, we assume that $\Co_n$ is a subspace of $\Mat_n(K)$, and define
$$
\Co_k := \{ C \in \Co_n \mid C e_{k+1} = C e_{k+2} = \cdots = C e_n = 0 \}
$$
for all $k < n$, where $e_i$ is the $i$-th standard basis unit vector. 

Define
\begin{align*}
d_k &:= \dim_{K(x)} \big((K(x) \otimes_K \Co_k) \cdot x \big) \\ 
 &\hphantom{:}= \dim_{K(x)} \sum_{C \in \Co_k} K(x) \cdot C \cdot x
\end{align*}
Notice that $0 = d_0 \le d_1 \le d_2 \le \cdots \le d_n = r+1$, where $r+1$ is as in theorems 
\ref{main2} and \ref{main3}. 

Lemma \ref{dimL} leads to the following corollary.

\begin{corollary} \label{vcol}
If $\#K \ge d_k$, then there exists a $v \in K^{\times n}$ with $v_{k+1} = v_{k+2} = \cdots = v_n = 0$,
such that $d_k = \dim_K \Co_k v$. If $\#K > d_k$, then we can additionally take $v_k = 1$.
\end{corollary}

\begin{proof}
The existence of a vector $v$ as claimed, except that $v_{k+1} = v_{k+2} = \cdots = 
v_n = 0$, follows from (ii) and (iii) of lemma \ref{dimL} respectively.
Since columns $k+1, k+2, \ldots, n$ of $\Co_k$ are zero, we can indeed take $v_{k+1} = v_{k+2} = \cdots 
= v_n = 0$. 
\end{proof}

\noindent
Unlike $d_n = r+1$, $d_k$ is not invariant under linear conjugation in general. But $d_k$ is indeed 
invariant under conjugation with lower triangular linear maps for every $k$, because of (i) of the 
following lemma.

\begin{lemma} \label{Tlem}
Suppose that the last $n-k$ columns of $T \in \GL_n(K)$ match those of a lower triangular matrix.
Then we have the following changes when we replace $\Co_n$ by $T^{-1} \Co_n T$.
\begin{enumerate}[{\upshape (i)}]
 
\item
$\Co_k e_k$ gets replaced by $T^{-1} \Co_k T e_k$ and $d_k$ stays the same.

\item
If the $k$-th column $T e_k$ of $T$ is zero above the diagonal, then $\Co_k e_k$ gets replaced by 
$T^{-1} \Co_k e_k$ and $\dim_K \Co_k e_k$ stays the same.

\end{enumerate}
Furthermore, we have the following for all $j > k$ when we replace $\Co_n$ by $T^{-1} \Co_n T$.
\begin{enumerate}[{\upshape (i)}] \addtocounter{enumi}{2}

\item
$\Co_j e_j$ gets replaced by $T^{-1} \Co_j e_j$ and $d_j$ and $\dim_K \Co_j e_j$ stay the same.

\item
If $B_{ij} = 1$ implies $T e_i = e_i$ for every $i$, then $B_{ij}$ will not change for any $i$.

\item
If $B_{ij} = 0$ implies $e_i\tp T = e_i\tp$ for each $i$ and $b_j = \dim_K \Co_j e_j$, then $B_{ij}$ 
will not change for any $i$.

\end{enumerate}
\end{lemma}

\begin{proof}
Since $T$ is lower triangular at the last $n-k$ columns, the last $n-k$ columns of $C \in \Mat_n(K)$ are 
zero, if and only if the last $n-k$ columns of $CT$ are zero, if and only if the last $n-k$ columns of
$T^{-1} C T$ are zero. Hence $\Co_k$ gets replaced by $T^{-1} \Co_k T$
when we replace $\Co_n$ by $T^{-1} \Co_n T$.
\begin{enumerate}[{\upshape (i)}]

\item
Since $\Co_k$ gets replaced by $T^{-1} \Co_k T$, the first claim is obvious. 
For a vector $v \in K(x)^{\times n}$, let $\phi(v) = T^{-1}v|_{x=Tx}$, where $|_{x=f(x)}$ means 
substituting $x$ by $f(x)$. Then $\phi^{-1}(v) = Tv|_{x=T^{-1}x}$, so $\phi$ is an isomorphism
between the spaces $(K(x) \otimes_K \Co_k) \cdot x$ and $T^{-1} \cdot (K(x) \otimes_K \Co_k) \cdot Tx$. 
In particular, the dimensions of these spaces are equal, which gives the second claim.

\item
Since $\Co_k e_k$ and $T^{-1} \Co_k e_k$ are isomorphic, the second claim follow from the first. 
Hence by (i), it suffices to show that $\Co_k T e_k = \Co_k e_k$.
For that purpose, assume that $T e_k$ is zero above the $k$-th coordinate.
Since $\Co_k$ in turn is zero at the right of the $k$-th column, only the $k$-th column of
$\Co_k$ and the $k$-th coordinate of $T e_k$ contribute to the product $\Co_k \cdot T e_k$, i.e.\@
$\Co_k \cdot T e_k = \Co_k e_k \cdot e_k\tp T e_k$.

The $k$-th coordinate $e_k\tp T e_k$ of $T e_k$ in nonzero, because $T \in \GL_n(K)$ is lower 
triangular at the last $n - k + 1$ columns. So we can cancel $e_k\tp T e_k$ to obtain
$\Co_k T e_k = \Co_k e_k$.

\item
Since $T$ is lower triangular at the last $n - j + 1$ columns, the desired results follow 
from (ii), (i) and (ii) respectively.

\item
Assume that $B_{ij} = 1$ implies $T e_i = e_i$ for all $i$.
We prove that $B_{ij}$ will not change for any $i$ by showing that $\Co_j e_j$ stays the same.
By (iii), $\Co_j e_j$ gets replaced by $T^{-1} \Co_j e_j$, so it suffices to show $(T^{-1} - I_n)
\Co_j e_j = 0$. If $B_{ij} = 0$, then the $i$-th coordinate of $\Co_j e_j$ is zero. 
If $B_{ij} = 1$, then the $i$-th column of $T^{-1} - I_n = T^{-1}(I_n - T)$ is zero by assumption. So 
$$
(T^{-1} - I_n) \Co_j e_j = (T^{-1} - I_n) I_n \Co_j e_j \subseteq 
\sum_{i=1}^n (T^{-1} - I_n)e_i \cdot e_i\tp  \Co_j e_j = 0
$$
indeed.

\item
Assume that $b_j = \dim_K \Co_j e_j$. By (iii), $\dim_K \Co_j e_j$ will 
stay the same, so by lemma \ref{small}, $b_j$ cannot decrease. So if some $B_{ij}$ changes, 
there will be an $i$ such that $B_{ij}$ changes from $0$ to $1$, which we assume from now on. 
We additionally assume that $B_{ij} = 0$ 
implies $e_i\tp T = e_i\tp$, so that $e_i\tp T^{-1} = e_i\tp$ as well. By (iii), 
$B_{ij} = \dim_K e_i\tp \Co_j e_j$  gets replaced by $\dim_K e_i\tp T^{-1} \Co_j e_j = 
\dim_K e_i\tp \Co_j e_j = B_{ij}$. So $B_{ij}$ will stay the same, which is a contradiction. \qedhere

\end{enumerate}
\end{proof}

\begin{proof}[Proof of theorem \ref{main3}.]
If $\dim_K \Co_j e_j = d_j$ for some $j$, then by (i) of lemma \ref{dimL} with $L = K(x_j)$ and 
$v = e_k + x_j e_j$, the condition of lemma \ref{zeroL} is satisfied for every $k$. 
Hence we will additionally 
arrange that $\dim_K \Co_j e_j = d_j$ for all $j$ by way of conjugation. As soon as we have
$\dim_K \Co_j e_j = d_j$ for some $j \ge 2$, it follows from (ii) of lemma \ref{zeroL} that 
$B_{ij} = 0$ implies $B_{i(j-1)} = 0$ for all $i$, so we do not need to show that any more.
\begin{enumerate}[(i)]
 
\item
{\em (Pass 1)}
We start with obtaining $\dim_K \Co_j e_j = d_j$ for all $j$. 
Suppose inductively that $\dim_{K} \Co_j e_j = d_j$ for all $j > k$ already. 
Since $d_k \le d_n = r+1 \le \#K$, it follows from corollary \ref{vcol} that there exists a 
$v \in K^{\times n}$ with $v_{k+1} = v_{k+2} = \cdots = v_n = 0$, such that $d_k = \dim_K \Co_k v$. 
Take $T \in \GL_n(K)$ such that $T e_k = v$ and $T e_j = e_j$ for all $j > k$. Then $T$ is
as in lemma \ref{Tlem}.

Now replace $\Co_n$ by $T^{-1} \Co_n T$. By (i) of lemma \ref{Tlem}, $d_k$ will not change, 
and $\Co_k e_k$ will become $T^{-1} \Co_k T e_k = T^{-1} \Co_k v$. Since $T^{-1} \Co_k v$ 
is isomorphic to $\Co_k v$, $\dim_K \Co_k e_k$ will become $\dim_K \Co_k v = d_k$. 
By (iii) of lemma \ref{Tlem}, $\dim_K \Co_j e_j = d_j$ will not be affected for any $j > k$. 

So we can obtain $\dim_K \Co_j e_j = d_j$ for all $j$ inductively.
The other claims of (i) follow as soon as we have $b_j = \dim_K \Co_j e_j = d_j$ for all $j$. 
We will arrange that by way of another induction pass.

{\em (Pass 2)}
Suppose inductively that $b_j = 
\dim_K \Co_j e_j = d_j$ for all $j > k$ already. We will obtain $b_k = \dim_K \Co_k e_k$ by way of a 
conjugation with a lower triangular matrix $T$. Just as above, the validity of $\dim_K \Co_j e_j = d_j$ 
for every $j > k$ will not be affected. But the validity of $\dim_K \Co_j e_j = d_j$ will not be affected 
for any other $j$ either, because $T$ is lower triangular at the last $n$ columns, see the proof of
(iii) of lemma \ref{Tlem}. 

Take a basis of $\Co_k e_k$ such that the positions of the first nonzero coordinates of the basis 
vectors are all different. Next, take $T \in \GL_n$ lower triangular, such that every column of $T$ 
is either one of those basis vectors of $\Co_k e_k$ (with its first nonzero coordinate on the diagonal of $T$)
or a standard basis unit vector (with its only nonzero coordinate on the diagonal of $T$), 
in such a way that all those basis vectors of $\Co_k e_k$ are included. 

Then $T^{-1}$ maps those basis vectors of $\Co_k e_k$ to standard basis unit vectors
(with corresponding positions of the first nonzero coordinate), 
so that $T^{-1} \Co_k e_k$ is spanned by standard basis unit vectors. By lemma \ref{small},
$b_k$ will become equal to $\dim_K \Co_k e_k$ when $\Co_k e_k$ gets replaced by $T^{-1} \Co_k e_k$.
Now replace $\Co_n$ by $T^{-1} \Co_n T$. By (ii) of lemma \ref{Tlem}, $\Co_k e_k$ will indeed be replaced
by $T^{-1} \Co_k e_k$, so that $b_k$ will become $\dim_K \Co_k e_k$. Furthermore, $\dim_K \Co_k e_k$ and $d_k$
will not change, so we indeed get $b_k = \dim_K \Co_k e_k = d_k$.

We prove that $b_j = \dim_K \Co_j e_j$ will not be affected by this conjugation 
for any $j > k$, by showing that $B_{ij}$ will not change for any $i$ and any $j > k$. 
By (v) of lemma \ref{Tlem}, it suffices to show that $B_{ij} = 0$ implies $e_i\tp T = e_i\tp$. 
So assume that $B_{ij} = 0$. Since the $i$-th row of $B$ is increasing, we have $B_{ik} = 0$ as well. 
Hence the $i$-th coordinate of any vector of $\Co_k e_k$ is zero. By construction of $T$, we have 
$e_i\tp T = e_i\tp$ indeed. So we can decrease $k$ and proceed.

\item
{\em (1 pass)} We start with the first step of the first induction pass in (i), to obtain
$\dim_K \Co_n e_n = d_n$.
As opposed to the double pass construction in (i), we will use a single pass construction here
to fulfill the claims of (i) and (ii) and the additional claim that $\dim_K \Co_j e_j = d_j$ for all $j$,
provided $\#K > \min\{d_{n-1},n-1\}$ after the first step of the first induction pass of (i) to obtain 
$\dim_K \Co_n e_n = d_n$. 

If $\#K \le \min\{d_{n-1},n-1\}$ after the first step of the first induction pass in (i), 
then we proceed with the double pass construction of (i), to obtain $b_{n-1} = \dim_K \Co_{n-1} 
e_{n-1} = d_{n-1}$. Since $d_{n-1}$ does not change any more after the first step of the first 
induction pass of (i), we get $\#K \le \min\{b_{n-1},n-1\}$, which implies (ii).

So assume that $\dim_K \Co_n e_n = d_n$ and $\#K > \min\{d_{n-1},n-1\}$. As long as $d_k = n$, 
we can proceed as in the first induction pass of (i) to obtain $b_j = \dim_K \Co_j e_j = d_j$
for all $j \ge k$, because by lemma \ref{small}, we have $b_k = n$ automatically if $\dim_K \Co_k e_k = n$. 
So suppose that $d_k \le n-1$ and that $b_j = \dim_{K} \Co_j e_j = d_j$ for all $j > k$.

{\em (Step 1)} 
We will first obtain $\dim_{K} \Co_k e_k = d_k$. If $k = n$, then we have already obtained
$\dim_{K} \Co_k e_k = d_k$.
So assume that $k \le n-1$. Then $d_k \le \min\{d_{n-1},n-1\} < \#K$. It follows from corollary \ref{vcol} that 
there exists a $v \in K^{\times n}$ with $v_{k+1} = v_{k+2} = \cdots = v_n = 0$, such that $d_k = \dim_K \Co_k v$ and
additionally $v_k = 1$. Make $T \in \GL_n$ by replacing the $k$-th column of $I_n$ by $v$.

Just as in the first pass of (i), we will obtain $\dim_K \Co_k e_k = d_k$ when we replace $\Co_n$ by $T^{-1} 
\Co_n T$. Furthermore, $d_k$ will not change, and neither will $d_j$ and $\dim_K \Co_j e_j$ for any $j > k$.
But as opposed to (i) and the case $k = n$, we have to show that $b_j = \dim_K \Co_j e_j$ will be preserved for 
all $j > k$, and that the rightmost $n-k$ columns of $B$ will stay decreasing above the diagonal. 
We do that by showing that the rightmost $n-k$ columns of $B$ will be preserved. 
For that purpose, take any column index $j > k$.

Since $T$ is just the identity matrix outside column $k$, it follows from (iv) of lemma \ref{Tlem} that
$B e_j$ will stay the same in case $B_{kj} = 0$. Hence assume that $B_{kj} = 1$. Then the induction assumption 
tells us that even $B_{1j} = B_{2j} = \cdots = B_{kj} = 1$. Since the last $n-k$ rows of $T$ are the same 
as those of $I_n$, it follows from (v) of lemma \ref{Tlem} that $B e_j$ will stay the same again. 
So let us proceed with replacing $\Co_n$ by $T^{-1} \Co_n T$. 

{\em (Step 2)} 
The next thing to arrange is that $b_k = \dim_K \Co_k e_k$, which can be done in the same manner 
as in the second induction pass of (i).

{\em (Step 3)}
At last, we must make the $k$-th column of $B$ decreasing above the diagonal. For that purpose, take 
$s < k$ maximal, such that $B_{sk} = 1$. Then there exists a permutation matrix $P$, which matches 
the identity matrix outside the leading principal minor matrix of size $s$, such that $P B e_k$ is 
decreasing above the $k$-th coordinate. Take $T = P^{-1}$. Then $P e_j = e_j = P^{-1} e_j = T e_j$
for all $j \ge k$, so $T$ satisfies both the condition of lemma \ref{Tlem} and the additional 
condition of (ii) of lemma \ref{Tlem}.

Now replace $\Co_n$ by $P \Co_n P^{-1} = T^{-1} \Co_n T$. By (i), (ii) and (iii) of lemma \ref{Tlem}, 
$\dim_k \Co_j e_j$ and $d_j$ will not change for any $j \ge k$. By (ii) of lemma \ref{Tlem},
$\Co_k e_k$ will be replaced by $T^{-1} \Co_k e_k = P \Co_k e_k$, and $B e_k$ will be replaced by 
$P B e_k$ along with it. So $B e_k$ will become decreasing above the $k$-th coordinate and $b_k$ 
stays the same. 

In order to prove that $B e_j$ will stay decreasing above the $j$-th coordinate 
and that $b_j$ will be maintained, for all $j > k$, we show that $B_{ij}$ stays the same for all
$i$ and all $j > k$.
By (v) of lemma \ref{Tlem}, it suffices to show that $B_{ij} = 0$ implies $e_i\tp T = e_i\tp$. 
If $i > s$, then $e_i\tp P =e_i\tp = e_i\tp P^{-1} = e_i\tp T$, so we may assume that $i \le s$.
By (ii) of lemma \ref{zeroL}, which is valid as long as $j > k$, we have 
$1 = B_{sk} = B_{s(k-1)} = \cdots = B_{sj}$.
Since $i \le s < k < j$ and $B e_j$ is decreasing above the $j$-th coordinate, $B_{ij} = 1$ 
is satisfied as well as $B_{sj} = 1$. Hence $B_{ij} = 0$ implies $e_i\tp T = e_i\tp$ once again.
So we can decrease $k$ and proceed.

\item Assume that $I_n \in \Co_n$. Since $b_{n-1} = \dim_K \Co_{n-1} e_{n-1}$ on 
account of (i), we deduce from (iii) of lemma \ref{zeroL} that $B_{(n-1)n} \ge B_{n(n-1)}$. 
So $b_n > \min\{b_{n-1},n-1\}$ remains to be proved. Hence assume that $b_n \le n-1$. Then
there exists an $i$ such that $B_{in} = 0$. By (ii) of lemma \ref{zeroL}, we have 
$e_i\tp \Co_{n-1} = \{0\}$. So 
$$
e_i\tp \cdot \Co_{n-1} \cdot (e_i  + x_{n-1}e_{n-1}) = 0 \ne 1 = e_i\tp \cdot I_n \cdot (e_i  + x_{n-1}e_{n-1})
$$
Consequently, we deduce from (i) of lemma \ref{zeroL} that
\begin{align*}
b_{n-1} &= \dim_K \Co_{n-1} e_{n-1} \\
        &\le \dim_{K(x_{n-1})} \big(K(x_{n-1}) \otimes_K \Co_{n-1}\big) \cdot (e_i  + x_{n-1}e_{n-1}) \\
        &< \dim_{K(x_{n-1})} \big(K(x_{n-1}) \otimes_K (\Co_{n-1} + KI_n)\big) \cdot (e_i  + x_{n-1}e_{n-1}) \\
        &\le \dim_{K(x_{n-1})} \big(K(x_{n-1}) \otimes_K \Co_n\big) \cdot (e_i  + x_{n-1}e_{n-1})
\end{align*}
From (i) of lemma \ref{dimL},
it follows that the right hand side does not exceed $\dim_{K(x)} \big((K(x) \otimes_K \Co_n) \cdot x\big)$.
So $b_{n-1} < d_n$. Since we arranged $b_n = \dim_K \Co_n e_n = d_n$, we have $b_n > b_{n-1} \ge 
\min\{b_{n-1},n-1\}$. \qedhere
\end{enumerate}
 
\end{proof}

\noindent
The double pass construction in (i) of the proof of theorem \ref{main3} is needed because the first 
induction pass may affect $b_j = \dim_{K} \Co_j e_j$. In the first step in (ii) of the proof of
theorem \ref{main3}, we additionally have $v_k = 1$, so that we can choose the transformation matrix
$T$ more conveniently than in the first induction pass in (i) of the proof of theorem \ref{main3}. 
Consequently, $b_j = \dim_{K} \Co_j e_j$ will not be affected in the first step in (ii) of the proof 
of theorem \ref{main3}, so that a double pass construction is not necessary there.

If $n = 3$ and $\Co$ is the space over $\F_2$ wich is spanned by
$$
\left( \begin{array}{ccc}
\bar{0} & \bar{1} & \bar{0} \\
\bar{0} & \bar{1} & \bar{0} \\
\bar{0} & \bar{0} & \bar{0}
\end{array} \right) \qquad \mbox{and} \qquad
\left( \begin{array}{ccc}
\bar{0} & \bar{0} & \bar{0} \\
\bar{0} & \bar{1} & \bar{1} \\
\bar{0} & \bar{0} & \bar{0}
\end{array} \right)
$$
then a computer calculation reveals that $\Co$ does not satisfy the claim of theorem \ref{main2}. 
This is because the condition of lemma \ref{zeroL} cannot be met. We use lemma \ref{vanish} 
to obtain this condition, but that requires a subset of cardinality three of $\F_2$.

\section{The radical of a Mathieu subspace of {\mathversion{bold}$\Mat_n(K)$}}

The results about the ideal $I$ in the preamble of the theorem below are well-known.
We have added the proof of these results for completeness only.

\begin{theorem} \label{rad}
Assume $\Ma$ is a $K$-subspace of $\Mat_n(K)$. As $\{0\} \subseteq \Ma$, 
we can take a left ideal $I$ of $\Mat_n(K)$ wich is contained in $\Ma$ and maximal
as such. Then $I$ is unique, has dimension 
$n k$ for some $k \le n$ and there exist a $T \in \GL_n(K)$ such that
$$
I T = T^{-1} I T = \{ M \in \Mat_n(K) \mid M e_{k+1} = M e_{k+2} = \cdots = M e_n = 0\}
$$
Furthermore, $I$ is a principal left ideal which is generated by an idempotent, and
the following statements are equivalent.
\begin{enumerate}[{\upshape (1)}]

\item $\Ma$ is a left Mathieu subspace of $\Mat_n(K)$,

\item $I$ contains all idempotents of $\Ma$,

\item $\rd(\Ma) = \rd(I)$.

\end{enumerate}
\end{theorem}

\begin{proof}
Since $\Ma$ is a $K$-subspace of $\Mat_n(K)$, the sum of two left ideals contained
in $\Ma$ is again contained in $\Ma$. Since $\Ma$ is a finite $K$-subspace of $\Mat_n(K)$,
we can deduce that $I$ is unique.

Take $M \in I$ of maximum rank $k$, and $T \in \GL_n(K)$ such that 
the last $n - k$ columns of $M T$ are zero. Since the first $k$ columns of 
$M T$ are independent of the last $n - k$ columns, the subspace of
$A \in \Mat_n(K)$ such that $A e_{k+1} = A e_{k+2} = \cdots = A e_n = 0$ is generated by $M T$
and therefore contained in $I T$. If $I T$ contains another matrix, then we get a contradiction with 
the maximality of $k$, because $I T$ is a left ideal of $M_n(K)$. 

Furthermore, $I T = T^{-1} I T$ is a principal left ideal which is generated by
$$
\left( \begin{array}{cc} I_k & \emptyset \\ \emptyset & \emptyset \end{array} \right)
$$
Hence $I$ is a principal left ideal which is generated by an idempotent as well. 
So it remains to show the following.
\begin{description}

\item[{\mathversion{bold}$(2) \Rightarrow (1)$}]
This follows from \cite[Th.\@ 4.2]{MR2859886}.  

\item[{\mathversion{bold}$(3) \Rightarrow (2)$}]
Suppose that $\rd(\Ma) = \rd(I)$. Since each idempotent of $\Ma$ is contained in $\rd(\Ma)$
and every idempotent in $\rd(I)$ is already in $I$, $I$ contains all idempotents of $\Ma$. 

\item[{\mathversion{bold}$(1) \Rightarrow (3)$}]
This follows from \cite[Lm.\@ 4.9]{MR2859886} or \cite[Th.\@ 4.10]{MR2859886}. \qedhere

\end{description}
\end{proof}

\begin{corollary} 
Suppose that $\Ma$ is a left Mathieu subspace of $\Mat_n(K)$, such that $0 < n^2 - \dim_K \Ma < n$.
Then $\Ma$ is even a two-sided Mathieu subspace of $\Mat_n(K)$ and $\#K > 2$. 
\end{corollary}

\begin{proof}
Take $I$ as in theorem \ref{rad}. We first prove that $\Ma$ is even two-sided.
On account of \cite[Th.\@ 4.2]{MR2859886}, it suffices to show that 
$\Ma$ has no nontrivial idempotent, which by $(1) \Rightarrow (2)$ of theorem \ref{rad} 
comes down to that $I$ has no nontrivial idempotents. Since theorem \ref{rad} additionally
tells us that $I$ is generated by a single idempotent, we just have to show that $I = (0)$. 

So assume that $I \ne (0)$. On account of theorem \ref{rad}, $I$ has dimension $n k$, where
$1 \le k \le n-1$ because $0 < n^2 - \dim_K \Ma$. Furthermore, we may assume that 
$I = \{ M \in \Mat_n(K) \mid M e_{k+1} = M e_{k+2} = \cdots = M e_n = 0\}$.

The space $\Ve$ defined by
$$
\left\{ M \in \Ma \,\left|\, M = \left( \begin{array}{cc} \emptyset & \tilde{M} \\ 
\emptyset & \lambda I_{n-k} \end{array} \right) \mbox{ for some } \tilde{M} \in \Mat_{k,n-k}(K) 
\mbox{ and a } \lambda \in K\right.\right\}
$$
is the intersection of $\Ma$ with a space of dimension $k(n-k) + 1$. Since the codimension of
$\Ma$ is less than $n \le k(n-k) + 1$, we have $\dim_K \Ve \ge 1$, so $\Ve$ has a nonzero element
$M$. If $\lambda \ne 0$ for $M$, then we take $E = \lambda^{-1} M$. If $\lambda = 0$ for $M$,
then we make $E$ from $M$ by replacing the leading principal minor matrix of size $k$ by $I_k$,
so that $E - M \in I$. In both cases, $E$ is an idempotent of $\Ma$ which is not contained in $I$.
This contradicts $(1) \Rightarrow (2)$ of theorem \ref{rad}, so $I = (0)$ indeed.

Next, we show that $\#K > 2$. Since the subspace of diagonal matrices of $\Mat_n(K)$ has dimension $n$ 
and $\Ma$ has codimension less than $n$, $\Ma$ contains at least two diagonal matrices, of which
one, say $E$, is nonzero. If $\#K = 2$, then $E$ is an idempotent and we have $E \in I$ because
$\Ma$ is a left Mathieu subspace of $\Mat_n(K)$. This contradicts $I = (0)$, so $\#K > 2$.
\end{proof}

\bibliographystyle{codimltn}
\bibliography{codimltn}

\end{document}